\documentclass[12pt]{amsart}
\usepackage{amsfonts,amssymb,amscd,amsmath,enumerate,verbatim,calc}

\newtheorem{Theorem}{Theorem}[section]
\newtheorem{Lemma}[Theorem]{Lemma}
\newtheorem{Corollary}[Theorem]{Corollary} 
\newtheorem{Proposition}[Theorem]{Proposition}

\newtheorem*{maintheorem}{Theorem}

\theoremstyle{definition}
\newtheorem{Definition}[Theorem]{Definition}

\newtheorem{chu}[Theorem]{}

\newtheorem*{chunk*}{}

\numberwithin{equation}{Theorem}
\newtheorem{Remark}[Theorem]{Remark}

\newcommand{\depth}{\operatorname{depth}}

\newcommand{\het}{\operatorname{ht}}

\newcommand{\Supp}{\operatorname{Supp}}
\newcommand{\Min}{\operatorname{Min}}
\newcommand{\Ann}{\operatorname{Ann}}
\newcommand{\Spec}{\operatorname{Spec}}
\newcommand{\fm}{{\mathfrak m}}
\newcommand{\fn}{{\mathfrak n}}
\newcommand{\fp}{{\mathfrak p}}

\newcommand{\fq}{{\mathfrak q}}

\begin{document}
%\subjclass[2010]{Primary 13. Secondary 13}
\title[Linear Formula]{A Linear Formula For the Generalized Multiplicity Sequence\\ }

\author[Dunn]{Thomas Dunn}
\address{Department of Mathematics 2750\\ North Dakota State University\\PO BOX 6050\\ Fargo, ND 58108-6050\\ USA}
\email{thomas.dunn@ndsu.edu}

\date{}
%\subjclass{13}

\begin{abstract}
For an arbitrary ideal $I$ in a local ring $R$ and a finitely generated $R$-module $M$, Achilles and Manaresi introduced the sequence of generalized multiplicities $c_{k}(I,M)$ ($k=0,...,\dim M$) as a generalization of the classical Samuel multiplicity $e(I,M)$ of an $\fm$-primary ideal $I$.  We prove a formula expressing each generalized multiplicity $c_{k}(I,M)$ as a linear combination of certain local multiplicities $c_{0}(IR_{\fp},M_{\fp})$.  As a consequence, when $M$ is formally equidimensional, we prove that if $I\subseteq J$ and $c_{k}(I,M)=c_{k}(J,M)$ for all $k=0,...,\dim M$ then $I$ is a reduction of $(J,M)$.  The converse of this statement is also known to be true by a result of Ciuperc\u{a}.  This theorem gives a complete numerical characterization of the integral closure, generalizing a well known theorem of Rees.
\end{abstract}

\maketitle
%\bigskip

\section{Introduction} 
Let $(R,\fm)$ be a local ring and $M$ a finitely generated $R$-module.  An ideal $I$ is said to be an ideal of definition on $M$ if the length $\lambda(M/IM)$ is finite.  In the case of $M=R$, being an ideal of definition on $R$ is equivalent to $I$ being $\fm$-primary.  For such an ideal the length $\lambda(M/I^{n}M)$ becomes eventually polynomial in $n$ of degree $d=\dim M$.  The leading coefficient of the  polynomial is $\frac{e(I,M)}{d!}$ where $e(I,M)$ is called the Hilbert-Samuel multiplicity of $I$ on $M$.  This multiplicity is a very important invariant in ideal theory.  Among other things, it gives a numerical characterization of the integral closure of an ideal, or equivalently, of reduction ideals.  For $I\subseteq J$ ideals, $I$ is said to be a reduction of $(J,M)$ if there exists some $n$ with $J^{n+1}M=IJ^{n}M$.  If $I$ is a reduction of $(J,M)$ and $I$ is an ideal of definition on $M$, then $e(I,M)=e(J,M)$.  The converse is true if $M$ is formally equidimensional as proved initially Rees \cite[Theorem 11.3.1]{HS} for the case $M=R$.  

In the case where $\lambda(M/IM)$ is not necessarily finite, consider the components of the bigraded module $G_{\fm}(G_{I}(M))$.  Achilles and Manaresi \cite{AM} considered the sum of the lengths of the bigraded components.  This sum eventually becomes a polynomial of degree $d=\dim M$.  The degree $d$ coefficients of this polynomial are given by $\frac{c_{k}(I,M)}{k!(d-k)!}$ where $c_{k}(I,M)$ is the $k$th generalized multiplicity of $I$ on $M$.  The multiplicity $c_{0}(I,M)$ coincides with the $j$-multiplicity $j(I,M)$ originally defined by Achilles and Manaresi \cite{AM2} as a generalization of the classical multiplicity $e(I,M)$.  The $j$-multiplicity is of particular interest for ideals of maximal analytic spread, which is the only case when it is nonzero.  Moreover, in the case when $\lambda(M/IM)$ is finite, the $j$-multiplicity coincides with the classical multiplicity $e(I,M)$.

There have been many attempts at proving generalizations of the theorem of Rees that give numerical characterizations of reduction ideals for arbitrary ideals.  Flenner and Manaresi proved that $I$ is a reduction of $(J,M)$ if and only if $j(I_{\fp},M_{\fp})=j(J_{\fp},M_{\fp})$ for every $\fp\in\Supp(M)$ \cite[Theorem 3.3]{FM}.  Ciuperc\u{a} proved if $I$ is a reduction of $(J,M)$ then $c_{k}(I,M)=c_{k}(J,M)$ for $k=0,...,d$ \cite[Theorem 2.7]{CC}.  

In this paper, we will prove the following.

\begin{maintheorem}
Let $(R,\fm)$ be a local ring, $I\subseteq R$ an ideal, $M$ a finitely generated $R$-module of dimension $d$ and denote
\[\Lambda_{k}=\Lambda_{k}(I,M)=\left\{\fp\mid\fp\in\Supp(M/IM), \dim(R/\fp)=k, \dim(R/\fp)+\dim M_{\fp} = \dim M\right\}.\]

Assume that the following condition is satisfied:
\[(\ast) \dim M_{\fp}=\dim(I^{n}M_{\fp})\text{ for all }n\geq0\text{ and for all }\fp\in\Supp(M/IM)\]
Then for all $k$=0,1,...,$d$ we have
\begin{equation}
c_{k}(I,M)=\sum_{\fp\in\Lambda_{k}}c_{0}(IR_{\fp},M_{\fp})e(R/\fp).
\end{equation}
\end{maintheorem}

We then show that the condition $(\ast)$ is satisfied whenever $\het_{M}(I)>0$.  As a consequence of the above equality, we prove the converse of Ciuperc\u{a}'s result.

\begin{maintheorem}
Let $(R,\fm)$ be a local ring, $M$ a finitely generated formally equidimensional module, and $I\subseteq J$ ideals with $\het_{M}(I)>0$.  If $c_{k}(I,M)=c_{k}(J,M)$ for all $k\leq d=\dim(M)$, then $I$ is a reduction of $(J,M)$.
\end{maintheorem}

\section{Preliminaries}
We begin with the following definitions.

\begin{Definition}
The height of $I$ in $M$ and analytic spread of $I$ in $M$ are defined as in \cite{FM}:
\[\het_{M}I:=\min\left\{\dim M_{\fp}\mid\fp\in\Supp M\cap V(I)\right\},\]
\[\ell_{M}(I):=\dim G_{I}(M)/\fm G_{I}(M).\]
\end{Definition}

For the generalized multiplicity sequence, we use the associated bigraded module $G_{\fm}(G_{I}(M)$.  The degree $(i,j)$ part of the bigraded module is given by \[\frac{I^{i}\fm^{j}M+I^{i+1}M}{I^{i}\fm^{j+1}M+I^{i+1}M}.\]

\begin{Definition}\cite{AM}
Consider the Hilbert function of the bigraded module $G_{\fm}(G_{I}(M))$
\[h_{I,M}(u,v)=\sum_{i=0}^{u}\sum_{j=0}^{v}\lambda\Big(\frac{\fm^{i}I^{j}M+I^{j+1}M}{\fm^{i+1}I^{j}M+I^{j+1}M}\Big).\]
The function $h_{I,M}(u,v)$ is a polynomial $p_{I,M}(u,v)$ for $u,v\gg0$ of degree $d=\dim M$.  The homogeneous degree $d$ part of $p_{I,M}(u,v)$ is 
\[\sum_{k+n=d}\frac{c_{k}(I,M)}{k!n!}u^{k}v^{n}\]
where $c_{k}(I,M)$ is the $k$th generalized multiplicity of $I$ on $M$.
\end{Definition}

The following is a version for modules of a result proved by Achilles and Manaresi.  This is particularly useful in that it gives the bounds for the vanishing multiplicities.

\begin{Proposition}\label{AM}
\cite[Theorem 2.3]{AM}.  Let $(R,\fm)$ be a local ring, $I\subset R$ an ideal, $M$ a finitely generated module, $G=G_{I}(R)$, $N=G_{I}(M)$.  Further, let $\ell=\ell_{M}(I)$ and $q=\dim(M/IM)$.  Then
\begin{enumerate}
	\item $c_{i}(I,M)=0$ for $i<d-\ell$ or $q<i\leq d$;
	\item $c_{d-l}(I,M)=\sum_{\beta}e(\fm G_{\beta},N_{\beta})e(G/\beta)$, where $\beta$ runs through all the highest dimensional associated primes of $N/\fm N$ such that $\dim(G/\beta)+\dim(N_{\beta})=\dim(N)$;
	\item $c_{q}(I,M)=\sum_{\fp}e(IR_{\fp},M_{\fp})e(R/\fp)$, where $\fp$ runs through all highest dimensional associated primes of $M/IM$ such that $\dim(R/\fp)+\dim(M_{\fp})=\dim(M)$.
\end{enumerate}
\end{Proposition}

It is also proved in \cite{AM} that $c_{0}(I,M)$ coincides with the so-called $j$-multiplicity $j(I,M)$, an invariant defined by Achilles and Manaresi in \cite{AM2}.  From \cite{AM2} $j(I,M)\neq0$ if and only if $I$ has maximal analytic spread, that is $\ell_{M}(I)=\dim M$.  Flenner and Manaresi proved the following result that relates the local $j$-multiplicities with reductions.
 
\begin{Theorem}\label{jm}
\cite[Theorem 3.3]{FM}. Let $(R,\fm)$ be a local ring and $M$ a formally equidimensional finitely generated module.  Let $I\subseteq J\subseteq \fm$ be ideals.  Then the following are equivalent:
\begin{enumerate}
\item $I$ is a reduction of $(J,M)$;
\item $j(IR_{\fp},M_{\fp})=j(JR_{\fp},M_{\fp})$ for all $\fp\in\Spec(R)$;
\item $j(IR_{\fp},M_{\fp})\leq j(JR_{\fp},M_{\fp})$ for all $\fp\in\Spec(R)$.
\end{enumerate}
\end{Theorem}

By \cite[Theorem 4.1]{McAdam}, we have that there are only finitely many primes with maximal analytic spread and therefore only finitely many primes where $j(I_{\fp},M_{\fp})\neq 0$.  The following result of Ciuperc\u{a} proves one direction generalization of the theorem of Rees.

\begin{Proposition} 
\cite[Proposition 2.7]{CC} Let $(R,\fm)$ be a local noetherian ring, $I\subseteq J$ be proper ideals of $R$, and $M$ a finitely generated $R$-module with $\dim M=d$.  If $I$ is a reduction of $(J,M)$, then $c_{k}(I,M)=c_{k}(J,M)$ for all $k=0,...,d$.
\end{Proposition}

\section{Superficial Elements}
In order to use induction on the dimension of the module, we require an element that preserves multiplicity and drops the dimension when we mod out by that element.

\begin{Definition}
For a given $r\in M\setminus\left\{0\right\}$, let $m$ be the largest number such that $r\in I^{m}M$.  The initial form $r^{*}$ is defined as the image of $r$ in $I^{m}M/I^{m+1}M$.  Then take the largest number $n$ where $r^{*}\in\fm^{n}(I^{m}M/I^{m+1}M)\subseteq G_{I}(M)$.  The initial form $r^{\prime}$ is defined as the image of $r^{*}$ in $(\fm^{n}I^{m}M+I^{m+1}M)/(\fm^{n+1}I^{m}M+I^{m+1}M)\subseteq G_{\fm}(G_{I}(M))$.  The initial form $I^{\prime}$ of an ideal $I$ is defined to be the ideal generated by the initial forms of all the elements in $I$.  
\end{Definition}

\begin{Definition}\label{supdef}
Let $S=G_{\fm}(G_{I}(R))$ and $N=G_{\fm}(G_{I}(M))$ and let $(0)=\bigcap_{i=0}^{t}N_{i}$ be an irredundant primary decomposition of the 0 submodule of $N$.  Denote $P_{i}=\sqrt{(N_{i}:_{S}N)}$ for $i=0,...,t$.  Assume that $I^{\prime}\subseteq P_{r+1},...,P_{t}$ and $I^{\prime}\not\subseteq P_{1},...,P_{r}$.  The element $x\in I$ is said to be a superficial element for $(I,M)$ if $x^{\prime}\notin P_{1},...,P_{r}$.
\end{Definition}

\begin{Proposition}\label{supexist}
Let $(R,\fm)$ be a local ring with infinite residue field, $M$ a finitely generated $R$-module, and $I\subset R$ an ideal.  Then there exists $x\in I\setminus\fm I$ such that $x$ is a superficial element for $(I,M)$.
\end{Proposition}

\begin{proof}
From Definition \ref{supdef}, an element $x$ is superficial if the initial form $x^{\prime}$ avoids finitely many primes $P_{i}$ (for $i=1,...,r)$ in the bigraded ring $G_{\fm}(G_{I}(R))$.  Since $I^{\prime}\not\subseteq P_{i}$, we have $Q_{i}=P_{i}\cap(I/\fm I)\subsetneq I/\fm I$; that is, each $Q_{i}$ is a proper subspace of the $(R/\fm)$ vector space $I/\fm I$.  Since $R/\fm$ is an infinite field, we have $(I/\fm I)\setminus\bigcup_{i=1}^{r}Q_{i}$ is nonempty.  Therefore, we may choose an element in $I\setminus\fm I$ whose image in $(I/\fm I)$ is not in $\bigcup_{i=1}^{r}Q_{i}$ as a superficial element.
\end{proof}

\begin{Remark}\label{suprem}
With the same assumptions as in Proposition \ref{supexist}; if $\depth_{I}(M)>0$, then there exists $x\in I\setminus\fm I$ superficial element that is a nonzero divisor on $M$.  
\end{Remark}

Since $\depth_{I}(M)>0$, we have $I\setminus\Ann(M)$ is nonempty.  Therefore, the subspace $((I\cap\Ann(M))+\fm I)/\fm I$ is a proper subspace of $I/\fm I$.  We may then set $Q_{0}=(I\cap\Ann(M))+\fm I)/\fm I$ and have $(I/\fm I)\setminus\bigcup_{i=0}^{r}Q_{i}$ is nonempty and proceed as above.

Since our goal is to use induction on the dimension of the module, we use the following proposition to ensure we can find a nonzero divisor.

\begin{Proposition}\label{nzd}
Let $(R,\fm)$ be a local ring, $M$ a finitely generated $R$-module, and $I\subseteq R$ an ideal.  If $\ell_{M}(I)>0$, then there exists $c$ such that $\depth_{I}(I^{c}M)>0$.
\end{Proposition}
\begin{proof}
Let $0=N_{1}\cap...\cap N_{t}$ be an irredundant primary decomposition of the $0$ submodule of $M$ and $\fp_{i}=\sqrt{(N_{i}:_{R}M)}$.  Further, assume that $I\not\subseteq\fp_{i}$ for $i=1,...,r$ and $I\subseteq\fp_{i}$ for $i=r+1,...,t$.  Let $x\in I$ be an element such that $x\notin\fp_{i}$ for $i=1,...,r$.  There exists $c$ such that $I^{c}M\subseteq\bigcap_{i=r+1}^{t}N_{i}$.  We also have $(0:_{M}x)=\bigcap_{i=1}^{t}(N_{i}:x)\subseteq\bigcap_{i=1}^{r}N_{i}$.  Therefore, $I^{c}M\cap(0:_{M}x)=0$.  Since $\ell_{M}(I)>0$, $I^{c}M\neq0$.  Therefore, $x\in I$ is a nonzero divisor on $I^{c}M$ hence $\depth_{I}(I^{c}M)>0$.
\end{proof}

\begin{Proposition}\label{superficial}
Let $(R,\fm)$ be a local ring with infinite residue field, $M$ a finitely generated module, and $I\subset R$ an ideal.  If $\ell_{M}(I)>0$ then there exists an integer $c$ and element $x\in I$ such that $x$ is a superficial element and a nonzero divisor on $I^{c}M$.
\end{Proposition}
\begin{proof}
This is a consequence of Proposition \ref{nzd} and Remark \ref{suprem}.
\end{proof}

The following result proved by Ciuperc\u{a}, shows that this choice of superficial element does indeed preserve the multiplicity sequence.

\begin{Proposition}\label{invariant2}
\cite[Theorem 2.11]{CC} Let $(R,\fm)$ be a local ring, $I\subset R$ an ideal, $M$ a finitely generated module, and $x\in I\setminus\fm I$ a superficial element for $(I,M)$ and a nonzero divisor on $M$.  Then $c_{i}(I,M)=c_{i}(I,M/xM)$ for $i\leq d-2$.
\end{Proposition}

Since we can choose a superficial element that is a nonzero divisor on $I^{n}M$ for $n\gg0$, we now show that the multiplicity sequence does not significantly change by replacing $M$ with $I^{n}M$ for $n$ large enough.

\begin{Proposition}\label{invariant1}
Let $(R,\fm)$ be a local ring, $M$ a finitely generated module, and $I\subset R$ an ideal.  If $\dim(I^{n}M)=\dim(M)=d$ for $n\gg 0$, then $c_{i}(I,M)=c_{i}(I,I^{n}M)$ for $i\leq d-1$ and $c_{d}(I,I^{n}M)=0$ for all $n\gg0$.
\end{Proposition}

\begin{proof}
Consider the Hilbert function: 

\[h_{I,M}(u,v)=\sum_{i=0}^{u}\sum_{j=0}^{v}\lambda\Big(\frac{\fm^{i}I^{j}M+I^{j+1}M}{\fm^{i+1}I^{j}M+I^{j+1}M}\Big).\]

The function $h_{I,M}(u,v)$ is eventually polynomial, $p_{I,M}(u,v)$, for $i,j\gg 0$.  The degree $d$ part of this polynomial $p(u,v)$ is:

\[\sum_{k+l=d}\frac{c_{k}(I,M)}{k!l!}u^{k}v^{l}.\]

We can now write $h_{I,I^{n}M}(u,v)$ as

\begin{tabular}{rl}
$h_{I,I^{n}M}(u,v)$&$=\displaystyle{\sum_{i=0}^{u}\sum_{j=0}^{v}\lambda(\fm^{i}I^{j+n}M+I^{j+n+1}M/\fm^{i+1}I^{j+n}M+I^{j+n+1}M)}$\\
&$\displaystyle{=\sum_{i=0}^{u}\sum_{j=n}^{v+n}\lambda(\fm^{i}I^{j}M+I^{j+1}M/\fm^{i+1}I^{j}M+I^{j+1}M)}$\\
&$=h_{I,M}(u,v+n)-h_{I,M}(u,n)$ which implies that\\
$p_{I,I^{n}M}(u,v)$&$=p_{I,M}(u,v+n)-p_{I,M}(u,n)$ for $u,v,n\gg0$.\\
\end{tabular}

Since the polynomials $p_{I,I^{n}M}(u,v)$ and $p_{I,M}(u,v+c)-p_{I,M}(u,c)$ are equal, we have an equality on the degree $d$ parts.  For a fixed $n$ large enough, the degree $d$ part of $p_{I,M}(u,v+n)$ is the same as the degree $d$ part $p_{I,M}(u,v)$ and the degree $d$ part of $p_{I,M}(u,n)$ is $\frac{c_{d}(I,M)}{d!}u^{d}$.\\ 

Since $\dim I^{n}M=d$, the degree of $p_{I,I^{n}M}(u,v)$ must be $d$ as well.  Therefore, we can conclude that $c_{k}(I,M)=c_{k}(I,I^{n}M)$ for $k\leq d-1$ and $c_{d}(I,I^{n}M)=0$.
\end{proof}

We need to show that we can find a superficial element that is compatible with various localizations.  This will be done by providing a correspondence between the subspaces of a vector space over various residue fields of the localizations and the subspaces of a vector space over the initial residue field. 

\begin{Lemma}\label{loc1}
Let $(R,\fm)$ be a local ring with infinite residue field $k$, $\fp$ a prime ideal, $(S,\fn)$ the local domain $R/\fp$, and $V$ a subspace of $Q(S)^{n}$.  The map $\Phi$ given by $\Phi(V)=((V\cap S^{n})+\fn S^{n})/\fn S^{n}$ has the following properties:
\begin{enumerate}
\item $\Phi(V)=k^{n}$ if and only if $V=Q(S)^{n}$.
\item If $(\overline{x_{1}},...,\overline{x_{n}})\notin\Phi(V)$ then $(\overline{x_{1}},...,\overline{x_{n}})\notin V$.
\end{enumerate}
\end{Lemma}

\begin{proof}
For (1), note that it is clear that if $V=Q(S)^{n}$, then $\Phi(V)=k^{n}$.  Suppose $\Phi(V)=k^{n}$.  Then $\Phi(V)=(V\cap S^{n}+\fn S^{n})/\fn S^{n}=S^{n}/\fn S^{n}$; therefore by Nakayama's Lemma, we have $(V\cap S^{n})+\fn S^{n}=S^{n}$ and $V\cap S^{n}=S^{n}$.  Thus $V\supseteq S^{n}$.  Since $V$ is a vector space that contains $S^{n}$ (in particular it contains the standard basis), $V=Q(S)^{n}$.

Part (2) is obvious, as if $(\overline{x_{1}},...,\overline{x_{n}})\in V$ then $(\overline{x_{1}},...,\overline{x_{n}})\in\Phi(V)$.
\end{proof}

\begin{Remark}\label{basis}
Let $(R,\fm)$ be a local ring with infinite residue field $k$ and $I\subset R$ an ideal, $I=(x_{1},...,x_{n})$.  Let $T:k^{n}\rightarrow I/\fm I$ defined by $T(\overline{a_{1}},...,\overline{a_{n}})=\overline{a_{1}x_{1}}+...+\overline{a_{n}x_{n}}$.  Then for finitely many proper subspaces $V_{i}$ $(i=1,...,r)$ of $I/\fm I$, there exists $(\overline{a_{1}},...,\overline{a_{n}})$ such that $T(\overline{a_{1}},...,\overline{a_{n}})\notin\bigcup_{i=1}^{r} V_{i}$.
\end{Remark}
\begin{proof}
Since $I/\fm I\cong k^{s}$ for some $s\leq n$, there exists $x\in I/\fm I$ that avoids $\bigcup_{i=1}^{r} V_{i}$.  By the surjectivity of $T$, there exists $(\overline{a_{1}},...,\overline{a_{n}})\in k^{n}$ such that $T(\overline{a_{1}},...,\overline{a_{n}})=x\notin\bigcup V_{i}$.
\end{proof}

\begin{Proposition}\label{superloc}
Let $(R,\fm)$ be a local ring with infinite residue field $k$, $I=(x_{1},...,x_{n})\subset R$ an ideal, $M$ a finitely generated module.  Then there exists $x$ such that $x\in I$ is superficial for $(I,M)$ and $\frac{x}{1}\in IR_{\fp}$ is superficial for $(IR_{\fp},M_{\fp})$ for all $\fp$ where $\ell_{M_{\fp}}(IR_{\fp})=\dim M_{\fp}$.
\end{Proposition}
\begin{proof}
By \cite[Theorem 4.1]{McAdam} there are only finitely many primes $\fp$ with $\ell_{M_{\fp}}(IR_{\fp})=\dim M_{\fp}$.  Let $\left\{\fp_{1},...,\fp_{t}\right\}$ be the finite set of primes with this property, denote $\fp_{0}=\fm$ and let $k_{i}=Q(R/\fp_{i})$ the field of fractions of $R/\fp_{i}$.  Let $\Phi_{i}$ be the map from subspaces of $k_{i}^{n}$ defined in Lemma \ref{loc1}.  By Remark \ref{basis} we may expand subspaces of $IR_{\fp_{i}}/\fp_{i}IR_{\fp_{i}}$ to $k_{i}^{n}$.  For each $i$, let $V_{(i,j)}\subsetneq k_{i}^{n}$ be the finitely many subspaces of $k_{i}^{n}$ such that $a_{1}x_{1}+...+a_{n}x_{n}\in IR_{\fp_{i}}\setminus\fp_{i}IR_{\fp_{i}}$ is superficial whenever $(\overline{a_{1}},...,\overline{a_{n}})\in k_{i}^{n}\setminus\bigcup_{j}V_{(i,j)}$.

By Lemma \ref{loc1}, for every $i,j$ we have that $\Phi_{i}(V_{(i,j)})$ is a proper subspace of $k^{n}$ and therefore $\bigcup_{i,j}\Phi_{i}(V_{(i,j)})$ is a union of finitely many proper subspaces of $k^{n}$.  If $(\overline{a_{1}},...,\overline{a_{n}})\in k^{n}\setminus\bigcup_{i,j}\Phi_{i}(V_{(i,j)})$ then $a_{1}x_{1}+...+a_{n}x_{n}\in IR_{\fp_{i}}\setminus\fp_{i}IR_{\fp_{i}}$ is superficial for $(IR_{\fp_{i}},M_{\fp_{i}})$ for all $i=0,...,t$.
\end{proof}
\section{The Main Result}

We will now give a generalization of the formula given by Theorem \ref{AM}$(3)$.

\begin{Theorem}\label{main}
Let $(R,\fm)$ be a local ring, $I\subseteq R$ an ideal, $M$ a finitely generated $R$-module of dimension $d$ and denote
\[\Lambda_{k}=\Lambda_{k}(I,M)=\left\{\fp\mid\fp\in\Supp(M/IM), \dim(R/\fp)=k, \dim(R/\fp)+\dim M_{\fp} = \dim M\right\}.\]

Assume that the following condition is satisfied:
\[(\ast) \dim M_{\fp}=\dim(I^{n}M_{\fp})\text{ for all }n\geq0\text{ and for all }\fp\in\Supp(M/IM)\]
Then for all $k$=0,1,...,$d$ we have
\begin{equation}\label{maineq}
c_{k}(I,M)=\sum_{\fp\in\Lambda_{k}}c_{0}(IR_{\fp},M_{\fp})e(R/\fp).
\end{equation}
\end{Theorem}
\begin{proof}
We may replace $R$ with the faithfully flat extension $R[x]_{\fm[x]}$ and $M$ with $M\otimes R[x]_{\fm[x]}$ and in doing so assume $R$ has an infinite residue field.  First note the sums in \ref{maineq} are finite.  Indeed, by Proposition \ref{AM} $c_{0}(IR_{\fp},M_{\fp})=0$ whenever $\dim M_{\fp}\neq\ell_{M_{\fp}}(IR_{\fp})$ and the set of primes $\fp$ that satisfy $\dim M_{\fp}=\ell_{M_{\fp}}(IR_{\fp})$ is finite \cite[4.1]{McAdam}.  

We will proceed by induction on $d=\dim M$.  For $d=0$, (\ref{maineq}) becomes $c_{0}(I,M)=c_{0}(I,M)e(R/\fm)$, which is trivial. 

Now consider the case $\dim M=1$.  It is clear that (\ref{maineq}) holds for $k=0$.  If $\dim(M/IM)=1$, then by Proposition \ref{AM}(3) we have \[c_{1}(I,M)=\sum_{\fp\in\Lambda_{1}}c_{0}(IR_{\fp},M_{\fp})e(R/\fp).\]If $\dim(M/IM)=0$, by Proposition \ref{AM}(1) we have $c_{1}(I,M)=0$.  It suffices to show that $\Lambda_{1}$ is empty.  Indeed, if $\fp\in\Supp(M/IM)$, then $\dim(R/\fp)\leq\dim(M/IM)=0$.

Now suppose that (\ref{maineq}) holds for all finitely generated $R$-modules of dimension less than $d$ that satisify $(\ast)$.  Let $M$ be a finitely generated $R$-module of dimension $d\geq 2$ that satisfies $(\ast)$.  

First we prove (\ref{maineq}) for $k=d$.  If $\dim(M/IM)<d$, then $\Lambda_{d}$ is empty.  (If $\fp\in\Supp(M/IM)$, then $\dim(R/\fp)<d$.)  Also, $c_{d}(I,M)=0$ by Proposition \ref{AM}(1).  If $\dim(M/IM)=d$, then by Proposition \ref{AM}(3) we get \[c_{d}(I,M)=\sum_{\fp\in\Lambda_{d}}c_{0}(IR_{\fp},M_{\fp})e(R/\fp).\]  From now on we assume $k\leq d-1$.

If $\ell_{M}(I)=0$, then $I^{n}M=0$ for $n\gg0$ and since $M$ satisfies $(\ast)$, we must have $\dim M=0$. Therefore $\ell_{M}(I)>0$.  By Proposition \ref{superficial} there exists a positive integer $c$ such that $\depth_{I}(I^{c}M)>0$.  We claim that if the theorem holds for the module $I^{c}M$ and $k\leq d-1$, then the theorem holds for the $R$-module $M$ and all $k\leq d-1$.

Indeed, if $M$ satisfies $(\ast)$, then $I^{c}M$ is also $d$-dimensional and satisfies $(\ast)$, so assume \[c_{k}(I,I^{c}M)=\sum_{\fp\in\Lambda_{k}}c_{0}(IR_{\fp},I^{c}M_{\fp})e(R/\fp)\] for $k=0,...,d-1$.  However, by Proposition \ref{invariant1} we have $c_{k}(I,I^{c}M)=c_{k}(I,M)$ for $k=0,...,d-1$ and $c_{0}(IR_{\fp},I^{c}M_{\fp})=c_{0}(IR_{\fp},M_{\fp})$ for $\fp\in\Lambda_{k}(I^{c}M)$ $(\dim M_{\fp}\geq 1$ because $d-1\geq k=\dim(R/\fp)=d-\dim M_{\fp})$.  Also note that $\Lambda_{k}(I,M)\supseteq\Lambda_{k}(I^{c}M)$ for $k\leq d-1$.  We only need to show that $\Lambda_{k}(I,M)\subseteq\Lambda_{k}(I,I^{c}M)$.

By contradiction, assume that $\fp\in\Lambda_{k}(I,M)\setminus\Lambda_{k}(I,I^{c}M)$ with $k\leq d-1$.  Then $I^{c}M_{\fp}=I^{c}M_{\fp}$ and by Nakayama's Lemma we get $I^{c}M_{\fp}=0.$  Then $\dim M_{\fp}=\dim I^{c}M_{\fp}=0$, and since $\fp\in\Lambda_{k}(I,M)$ we obtain $k=\dim R/\fp = \dim M$, which is a contradiction.  This concludes the proof of the claim.  So by replacing $M$ with $I^{c}M$ we may assume that $\depth_{I}(M)>0$.

Next we consider the case when $k=d-1$.  If $\dim(M/IM)=d-1$, then by Proposition \ref{AM}(3) it follows that (\ref{maineq}) holds for $k=d-1$.  If $\dim(M/IM)<d-1$, then by Proposition \ref{AM}(1) we have $c_{d-1}(I,M)=0$.  We also note that in this case $\Lambda_{d-1}(M)=\emptyset$.  If $\fp\in\Supp(M/IM)$, then $\dim(R/\fp)\leq\dim(M/IM)<d-1$.  So we may assume that $k\leq d-2$.

Let $x\in I$ be a nonzero divisor on $M$ that is superficial on $M$ and $M_{\fp}$ for all the finitely many prime ideals $\fp$ such that $c_{0}(IR_{\fp},M_{\fp})\neq 0$. Such an element exists by \ref{superloc}.  Note that $M/xM$ is a $d-1$ dimensional $R$-module that also satisfies $(\ast)$.  By the induction hypothesis \[c_{k}(I,M/xM)=\sum_{\fp\in\Lambda_{k}(M/xM)}c_{0}(IR_{\fp},(M/xM)_{\fp})e(R/\fp).\]Note that $\Lambda_{k}(I,M)=\Lambda_{k}(I,M/xM)$.  

By Proposition \ref{invariant2} $c_{k}(I,M/xM)=c_{k}(I,M)$ and $c_{0}(IR_{\fp},(M/xM)_{\fp})=c_{0}(IR_{\fp},M_{\fp})$ for $\dim M_{\fp}\geq 2$ which is true for 
$\fp\in\Lambda_{k}(I,M)$ and $k\leq d-2$.  Therefore \[c_{k}(I,M)=\sum_{\fp\in\Lambda_{k}(I,M)}c_{0}(IR_{\fp},M_{\fp})e(R/\fp)\] for all $k\leq d-2$.  This finishes our inductive argument.
\end{proof}

The condition $(\ast)$ was needed to be able to replace $M$ with $I^{n}M$ without changing the multiplicity sequence.  However, we note that $(\ast)$ is satisfied whenever $\het_{M}(I)>0$.

\begin{Corollary}\label{cor}
Let $(R,\fm)$ be a local ring, $I\subset R$ an ideal, and $M$ a finitely generated $R-$module.  If $\het_{M}(I)>0$, then for $0\leq k\leq d=\dim(M)$ we have
\[c_{k}(I,M)=\sum_{\fp\in\Lambda_{k}}c_{0}(IR_{\fp},M_{\fp})e(R/\fp).\]
\end{Corollary}
\begin{proof}
If suffices to show that if $\het_{M}(I)>0$, then $\dim M_{\fp}=\dim I^{c}M_{\fp}$ for all positive integers $c$ and all primes $\fp\in\Supp(M/IM)$.

If $\het_{M}(I)>0$, then $I\not\subseteq\fq$ for all $\fq\in\Min(M)$.  Since $\het_{M}(I)>0$, $I$ is not in any minimal prime of $M$ and so $IR_{\fp}$ is not in any minimal prime of $M_{\fp}$.  Therefore, $\het_{M_{\fp}}(IR_{\fp})>0$ and it suffices to show that $\dim(M_{\fp})=\dim(I^{n}M_{\fp})$ only for $\fp=\fm$.

Since $I^{n}M\subseteq M$, we have $\sqrt{\Ann(M)}\subseteq\sqrt{\Ann(I^{n}M)}$.  Let $y\in\Ann(I^{n}M)$.  We have $yI^{n}M=0$ and so $yI^{n}\in\Ann(M)$.  So $yI^{n}\in\fq$ for every $\fq\in\Min(M)$.  Since $I\not\subseteq\fq$, we obtain $y\in\fq$ and so $\sqrt{\Ann(M)}=\sqrt{\Ann(I^{n}M)}$.  Therefore $\dim(M)=\dim(I^{n}M)$.   
\end{proof}

\section{Multiplicity and Reduction}

The primary reason for attaining the formula in Theorem \ref{main} is to prove the following converse of \cite[Theorem 2.7]{CC}.

\begin{Theorem}\label{maincor} 
Let $(R,\fm)$ be a local ring, $M$ a finitely generated formally equidimensional module, and $I\subseteq J$ ideals with $\het_{M}(I)>0$.  If $c_{k}(I,M)=c_{k}(J,M)$ for all $k\leq d=\dim(M)$, then $I$ is a reduction of $(J,M)$.
\end{Theorem}
\begin{proof}
Using the definition of $\Lambda_{k}(I,M)$ given in Theorem \ref{main}, since $M$ is assumed to be formally equidimensional, we have \[\Lambda_{k}(I,M)=\left\{\fp\mid\fp\in\Supp(M/IM), \het_{M} \fp=d-k\right\}.\]  
Further, note that $\Lambda_{k}(J,M)\subseteq\Lambda_{k}(I,M)$.  By Proposition \ref{jm} we need to show $c_{0}(IR_{\fp},M_{\fp})=c_{0}(JR_{\fp},M_{\fp})$ for all $\fp\in\Spec(R)$.  Note that for a prime ideal $\fp\not\in\Supp(M/IM)$, we have $IR_{\fp}=JR_{\fp}=R_{\fp}$ or $M_{\fp}=0$, and so $IR_{\fp}$ is a reduction of $(JR_{\fp},M_{\fp})$.  Therefore, we only need to show $c_{0}(IR_{\fp},M_{\fp})=c_{0}(JR_{\fp},M_{\fp})$ for all $\fp\in\Supp(M/IM).$  Further, note that $\Lambda_{k}(I,M)=\emptyset$ for $k>q=\dim(M/IM)$ and $\Supp(M/IM)=\bigcup_{i=0}^{q}\Lambda_{i}(I,M)$.  We will proceed by induction on $\het_{M}(\fp)$ for $\fp\in\Supp(M/IM)$.  

First, we show equality for all the prime ideals with $\het_{M} \fp=d-q$, equivalently $\fp\in\Lambda_{q}(I,M)$.

Since $c_{q}(I,M)=c_{q}(J,M)$, by Theorem \ref{main} we have \[\sum_{\fp\in\Lambda_{q}(I,M)}c_{0}(IR_{\fp},M_{\fp})e(R/\fp)=\sum_{\fp\in\Lambda_{q}(J,M)}c_{0}(JR_{\fp},M_{\fp})e(R/\fp).\]

Note that $\Lambda_{q}(I,M)\supseteq\Lambda_{q}(J,M)$.  For $\fq\subsetneq\fp\in\Lambda_{q}(I,M)$, clearly $IR_{\fq}$ is a reduction of $(JR_{\fq},M_{\fq})$ as $\fq\not\in\Supp(M/IM)$.  Therefore, by \cite[Theorem 3.2]{FM} we have $c_{0}(IR_{\fp},M_{\fp})\geq c_{0}(JR_{\fp},M_{\fp})$ for all $\fp\in\Lambda_{q}(I,M)$.

Since $c_{0}(IR_{\fp},M_{\fp})>0$ for $\fp\in\Lambda_{q}(I,M)$ (it coincides with the classical Samuel multiplicity), we have $\Lambda_{q}(I,M)=\Lambda_{q}(J,M)$ and $c_{0}(IR_{\fp},M_{\fp})=c_{0}(JR_{\fp},M_{\fp})$ for all $\fp\in\Lambda_{q}(I,M)$.  By Theorem \ref{jm} we have $IR_{\fp}$ is a reduction of $(JR_{\fp},M_{\fp})$. 

Suppose $c_{0}(IR_{\fq},M_{\fq})=c_{0}(JR_{\fq}M_{\fq}$ for all $\fq$ with $\het_{M}\fq<n\leq d$, equivalently by Theorem \ref{jm} $IR_{\fq}$ is a reduction of $(JR_{\fq},M_{\fq})$ for all $\fq$ with $\het_{M}\fq<n\leq d$.  Since $c_{d-n}(I,M)=c_{d-n}(J,M)$, by Theorem \ref{main} we have 
\begin{equation}\label{coreq}
\sum_{\fp\in\Lambda_{d-n}(I,M)}c_{0}(IR_{\fp},M_{\fp})e(R/\fp)=\sum_{\fp\in\Lambda_{d-n}(J,M)}c_{0}(JR_{\fp},M_{\fp}).
\end{equation}
For every $\fq\subsetneq\fp$, we have $IR_{\fq}$ is a reduction of $(JR_{\fq},M_{\fq})$ since $\het_{M}\fq<\het_{M}\fp$. By \cite[Theorem 3.2]{FM} $c_{0}(IR_{\fp},M_{\fp})\geq c_{0}(JR_{\fp},M_{\fp})$ for $\fp\in\Lambda_{d-n}(J,M)$.

If $\fp\in\Lambda_{d-n}(J,M)$ and $c_{0}(JR_{\fp},M_{\fp})\neq0$ then $c_{0}(IR_{\fp},M_{\fp})\neq0$ and so $\fp\in\Lambda_{d-n}(I,M)$ by the previous inequality.  If $\fp\in\Lambda_{d-n}(J,M)$ and $c_{0}(JR_{\fp},M_{\fp})=0$, then there exists some $\fq\subsetneq\fp$ such that $c_{0}(JR_{\fq},M_{\fq})\neq0$ and thus $c_{0}(IR_{\fq},M_{\fq})\neq0$, since $IR_{\fq}$ is a reduction of $(JR_{\fq}, M_{\fq})$, and $\fq\in\Supp(M/IM)$.  Since $\fq\subsetneq\fp$, we have  $\fp\in\Lambda_{d-n}(I,M)$.  Therefore we have $\Lambda_{d-n}(I,M)=\Lambda_{d-n}(J,M)$.

So we have from (\ref{coreq})
\[\sum_{\fp\in\Lambda_{q}(J,M)}c_{0}(IR_{\fp},M_{\fp})e(R/\fp)=\sum_{\fp\in\Lambda_{q}(J,M)}c_{0}(IR_{\fp},M_{\fp})e(R/\fp).\]

Therefore, $c_{0}(IR_{\fp},M_{\fp})=c_{0}(JR_{\fp},M_{\fp})$ for $\fp\in\Lambda_{d-n}(I,M)$.  By Theorem \ref{jm} we have $IR_{\fp}$ is a reduction of $(JR_{\fp},M_{\fp})$ for $\fp$ with $\het_{M} \fp=n$ which concludes the proof.
\end{proof}

\end{document}